\documentclass{amsart}
\usepackage{amsfonts,amsmath,amscd,amssymb,amsthm,newlfont}
\newtheorem{theorem}{Theorem}[section]
\newtheorem{lemma}[theorem]{Lemma}
\newtheorem{proposition}[theorem]{Proposition}
\newtheorem{corollary}[theorem]{Corollary}
\theoremstyle{definition}

\newtheorem{example}[theorem]{Example}
\newtheorem{remark}[theorem]{Remark}

\numberwithin{equation}{section}

\begin{document}

\title[Exterior degree of infinite groups]{Exterior degree of infinite groups}

\author[R. Rezaei]{Rashid Rezaei}
\address{Department of Mathematics, Malayer University, Malayer, Iran} 
\email{ras$\_$rezaei@yahoo.com}

\author[F.G. Russo]{Francesco G. Russo}
\address{DIEETCAM, Universit\'a Degli Studi di Palermo, Viale Delle Scienze, Edificio 8, 90128, Palermo, Italy} 
\email{francescog.russo@yahoo.com}


\date{\today}



\begin{abstract}
The  exterior degree of a finite group has been introduced in [P. Niroomand and R. Rezaei, On the exterior degree of finite groups, Comm. Algebra 39 (2011), 335--343] and the present paper is devoted to study the exterior degree of infinite groups. We find some inequalities of combinatorial nature, which generalize those  of the finite case and allow us to get structural restrictions for the whole group. 
\end{abstract}

\maketitle

\section{Introduction}

The structure of a finite group may be strongly restricted, once we have information about  invariants which are related to the number of  elements  satisfying a given property. For instance,  the property of being commutative has motivated some authors to introduce the so--called \textit{commutativity degree} of a finite group $E$,  defined as the ratio
\[\mathrm{d}(E) = \frac{|\{(x, y) \in E \times E  \ | \ [x, y] = 1\}|}{|E|^2} = \frac{1}{|E|^2} \underset{x \in E} {\sum} |C_E(x)|= \frac{k(E)}{|E|},\]
where $k(E)$ is the number of the $E$--conjugacy classes $[x]_E=\{x^g \ | \ g \in E\}$ that constitute $E$. There is a  wide production on $\mathrm{d}(E)$ and its generalizations in the last decades and we recall here \cite{ahmad1, ahmad2, ahmad3, ahmad4, rashid1, rashid2}. The \textit{exterior degree} of $E$  is a more recent invariant, studied for similar purposes by the second author in \cite{n2} and with combinatorial techniques in \cite{ahmad6, n3, n5, n6}. The exterior degree of $E$  is defined as
\[\mathrm{d}^\wedge(E)=\frac{|\{(x,y)\in E \times E \ | \  x \wedge y =1_{_{E \wedge E}}\}|}{|E|^2}= \frac{1}{|E|} \sum^{k(E)}_{i=1} \frac{|C^\wedge_E(x_i)|}{|C_E(x_i)|}, \,\,\, (*)\]
where the last equality is proved in \cite[Lemma 2.2]{n2}, and one of the main results in \cite{n2} shows that \[\mathrm{d}^\wedge(E) \le \mathrm{d}(E)\] so that the commutativity degree and the exterior degree are connected.

Since there is literature for the commutativity degree  of infinite groups (see again \cite{ahmad2, ahmad3, ahmad4, rashid1, rashid2}), it is natural to look for a corresponding treatment of the exterior degree. To the best of our knowledge, this precise point  is not available in literature and has motivated us to write the present paper. One of the main difficulties is due to the fact that we need of a meaningful topological structure over infinite groups, which should agree with the discrete topology of the finite case. This forced us to look for projective limits of finite groups, where we have a good theory of measure and may proceed by analogy with the finite case in a compatible way with respect to the topology which we will consider.

Section 2 is devoted to justify the use of projective limits of finite groups in our context and we will introduce some technical notions which will be helpful in Section 3, where the main results are placed. The terminology and the notations will follow  those of \cite{ahmad1, n1, n2, n3, n4, n5, n6, rashid1, rashid2}.

\section{Preliminaries}

Eick  \cite{eick2} has recently described the low dimensional homology  of an \textit{infinite pro--$p$--group of finite coclass and central exponent} $t \ge 1$ \[C_t=\mathbb{Z}^{d_t}_p \rtimes C_{p^t}=\langle t_1,\ldots,t_{d_t} \rangle \rtimes \langle g \rangle =\langle g,t_1,\ldots,t_{d_t} \ | \  g^{p^t}=1,\]\[ g^{-1}t_1g=t^{-1}_{d_t},  g^{-1}t_ig= t_{i-1}t^{-e_i}_{d_t} \ (1<i \le d_t), [t_i,t_j]=1 \ (1\le j<i \le d_t) \rangle,\] where $p$ is a prime, $e_i=1$ if $p^{t-1}$ divides $i-1$, $e_i=0$ if $p^{t-1}$ does not divide $i-1$,  $\mathbb{Z}^{d_t}_p$ is the direct product of $d_t=p^{t-1}(p-1)$ copies of the group $\mathbb{Z}_p$ of $p$--adic integers, $C_{p^t}$ is a cyclic group of order $p^t$ acting uniserially on $\mathbb{Z}^{d_t}_p$.

The above presentation of $C_t$ can be found in \cite[Proof of Theorem 7]{eick2} (see also  \cite[\S7]{charles}). The relevance of $C_t$ is emphasized by \cite[Theorem 7.4.12, Corollary 7.4.13]{charles}, which characterize an arbitrary infinite pro--$p$--group of finite coclass to be constructed always as $C_t$.

 Before to proceed, it is good  to recall some basic notions on pro--$p$--groups from \cite{charles}. On a compact (Hausdorff) group $G$ it is possible to introduce the filter basis \[\mathcal{P} (G)=\{N=\overline{N} \triangleleft G \ | \ G/N \ \mathrm{is \ a \ finite \ }  p-\mathrm{group}\}\] and $G$ is said to be a pro--$p$--group if \[G=\lim_{N \in \mathcal{P}(G)} G/N,\] that is, if $G$ is a projective limit of finite $p$--groups (see \cite[Definitions 7.1.12, 7.2.1, 7.2.3, 7.2.4]{charles}). Here the open subgroups of $G$ are exactly those closed subgroups of $p$--power index \cite[Lemma 7.2.2]{charles}. Of course, the topology of $G$ is  the unique topology induced by $\mathcal{P}(G)$. The notion of coclass of a pro--$p$--group can be found in \cite[Definition 7.4.1]{charles} and is originally due to Leedham--Green and Newman:  a finite $p$--group of order $p^n$ and nilpotency class $c$ has \textit{coclass r} if $n-c=r$; a pro--$p$--group $G$ has coclass $r$ if there exists some $u\ge2$ such that $G/\gamma_i(G)$ has coclass $r$ for all $i\ge u$, where $\gamma_i(G)$ denotes the $i$--th term of the lower central series of $G$.

From \cite{charles}, the  \textit{rank} of  a pro--$p$--group $G$ is defined by  \[\mathrm{rk}(G)= {\underset{H= \overline{H} \leq G} {\mathrm{sup}}} d(H),\] where $d(H)$ denotes the minimal number of elements which are necessary to  generate topologically $H$.  
If $G$ is a  torsion--free pro--$p$--group, then  $l=\mathrm{rk}(G)=\mathrm{tf}(G)$ is called \textit{torsion--free rank}.  
Now we know from \cite[p.148]{eick2} (see also \cite[\S 9]{charles}) that the second homology group $H_2(G,\mathbb{Z}_p)$ (with coefficients in $\mathbb{Z}_p$) of an infinite pro--$p$--group $G$  of finite coclass  is an abelian pro--$p$--group   of the form $H_2(G,\mathbb{Z}_p)=T(G) \times F(G)$, where $T(G)$ is a finite $p$--group and $F(G)\simeq \mathbb{Z}^l_p$.  As usual in these situations, $H_2(G,\mathbb{Z}_p)$ is called \textit{Schur multiplier} of  $G$ and the following bound was proved few years ago.
\begin{theorem}[See  \cite{eick2}, Theorem A] \label{bettina} An infinite pro--p--group $G$ of finite coclass and central exponent $t$ has $\mathrm{tf} (H_2(G,\mathbb{Z}_p))=\frac{1}{2} d_t$ for all $p>2$ and $t>1$. In particular, this is true for $G=C_t$. 
\end{theorem}
The case $p=2$  of  Theorem \ref{bettina} needs to be treated separately (see \cite[p.148 and \S 7]{eick2}). Some of our main results deal with generalizations of  Theorem \ref{bettina} and \cite[Theorems 3.2, 3.3]{moravec}.

From \cite{moravec}, the \textit{complete nonabelian tensor square} $G \widehat{\otimes} G$ of a pro--$p$--group $G$, which  generalizes the   \textit{nonabelian tensor square} in \cite{brown}, is the group topologically generated by the symbols $x\widehat{\otimes }y$, subject to the relations $xy\widehat{\otimes} z=(x^y\widehat{\otimes} z^y)(y\widehat{\otimes} z)$ and $y\widehat{\otimes}tz=(y\widehat{\otimes} z)(y^z\widehat{\otimes} t^z)$ for all $x,y,z,t \in G$, where $x^y=y^{-1}xy$ and so on for $z^y,y^z, t^z$. It is straightforward to check that $G \widehat{\otimes} G$ is a pro--$p$--group. The subgroup $\widehat{\nabla}(G)=\overline{\langle x\otimes x \ | \  x \in G \rangle}$ is a closed central subgroup of $G \widehat{\otimes} G$ and the quotient group \[G \widehat{\otimes} G/ \widehat{\nabla}(G)= G \widehat{\wedge} G\] is called \textit{complete nonabelian exterior square} of the pro--$p$--group $G$. In analogy with the homological methods which appear in  \cite{brown, n1},  the maps \[\widehat{\kappa} : x\widehat{\otimes}y \in G  \widehat{\otimes} G \mapsto [x,y]\in G' \ \mathrm{and} \ \widehat{\kappa'} : x\widehat{\wedge}y \in G  \widehat{\wedge} G \mapsto [x,y]\in G',\]  are epimorphisms of pro--$p$--groups such that $\ker  \widehat{\kappa} \simeq \widehat{\nabla}(G)$ and $\ker \widehat{\kappa'}\simeq H_2(G,\mathbb{Z}_p)$. We note that $H_2(G,\mathbb{Z}_p)$ is a closed central subgroup of $G  \widehat{\wedge} G$. Also  \[\widehat{\varepsilon} : x \widehat{\otimes}y \in G \widehat{\otimes} G \mapsto  x \widehat{\wedge} y  \in G \widehat{\wedge} G \ \mathrm{and} \  \widehat{\varphi} :  x \widehat{\otimes} x \in \widehat{\nabla}(G) \mapsto x \widehat{\wedge} x \in H_2(G,\mathbb{Z}_p)\] 
are epimorphisms of pro--$p$--groups. More details on  $\widehat{\kappa}$, $\widehat{\kappa'}$, $\widehat{\varepsilon}$ and $\widehat{\varphi}$ can be found in \cite[\S 2]{moravec}. 
In virtue of all we have said,   the famous diagram in \cite{brown} 
 becomes the following, whose rows are central extensions of pro--$p$--groups.
\[\begin{CD}
1 @>>>\widehat{\nabla}(G)@>>> G\widehat{\otimes} G @>\widehat{\kappa}>> G' @>>> 1 \hspace{1.5cm}\\
@. @V\widehat{\varphi}VV @V\widehat{\varepsilon}VV @| \hspace{1.5cm} (**)\\
1 @>>> H_2(G,\mathbb{Z}_p)  @>>>G\widehat{\wedge} G @>\widehat{\kappa'}>> G' @>>>1 \hspace{1.5cm}\\
\end{CD}
\]
We have information on the topology of $G \widehat{\wedge} G$ by the next result.
\begin{theorem} [See  \cite{moravec}, Theorem 2.1]\label{primoz}
If $G=\lim_{N \in \mathcal{P}(G)}G/N$ is a pro--$p$--group, then  $G \widehat{\otimes} G =\lim_{(N,M) \in \mathcal{P}(G) \times \mathcal{P}(G)} G/N \otimes G/M$  is a pro--$p$--group. The same is true  for $G \widehat{\wedge} G$. 
\end{theorem}
Roughly speaking, Theorem \ref{primoz} shows a  passage under projective limit for the operator $\widehat{\wedge}$ and ensures a topological structure on $G  \widehat{\wedge}  G$. This allows us to generalize the  \textit{exterior centralizer} $C_E^\wedge(x)$  of an element $x$ of a finite group $E$ and the \textit{exterior center}  $Z^\wedge(E)$, recently studied in \cite{n1, n2, n3, n4, n5, n6}. The interest for $C_E^\wedge(x)$ and $Z^\wedge(E)$ is due to the fact that they provide criteria to decide if we have a \textit{capable group} or not  (see \cite{bfs, n1, n2}), that is, if our group is isomorphic or not to the inner automorphism of another group.

Some basic properties, studied in \cite{brown, n1} for finite groups, may be adapted to the infinite case in the sense of the next result.

\begin{lemma} \label{l:1}
Let $G$ be a pro--p--group and $x,y,z,t \in G$. Then
\[(x^{-1} \widehat{\otimes} y)^x = (x \widehat{\otimes} y)^{-1} = (x \widehat{\otimes} y^{-1})^y ;  \,\,\,\,\,\,\,\,\,\,
 (z \widehat{\otimes} t)^{xy} \ (g \widehat{\otimes} y)=(g \widehat{\otimes} y) \ (z \widehat{\otimes} t)^{yx} ; \]
\[ z \widehat{\otimes} (y^x y^{-1})= (x \widehat{\otimes} y)^ z \ (x \widehat{\otimes} y)^{-1} ; \,\,\,\,\,\,\,\,\,\,
 (x (x^{-1})^y) \widehat{\otimes} t=(x \widehat{\otimes} y) \ ((x \widehat{\otimes} y)^{-1})^t ; \]
\[ (z \widehat{\otimes} t)^{(x \widehat{\otimes} y)} =(z \widehat{\otimes} t)^{[x,y]}; \,\,\,\,\,\,\,\,\,\,
 [x \widehat{\otimes} y, z \widehat{\otimes} t]= (x (x^{-1})^y)(t^z t^{-1}). \]
The same rules are true if we replace $\widehat{\wedge}$ with $\widehat{\otimes}$.
\end{lemma}

\begin{proof} It is enough to repeat the computations in \cite[Propositions 1,2,3]{brown}, replacing the role of $\otimes$ with that of $\widehat{\otimes}$ (resp., $\wedge$ with $\widehat{\wedge}$).
\end{proof}

The \textit{complete exterior centralizer} of a pro--$p$--group $G$ is the set
\[\widehat{C_G}(x)=\{a\in G  \ | \ a \widehat{\wedge} x=1_{_{G \widehat{\wedge} G}}\}\] 
which turns out to be a subgroup of $G$, because  if $a,b \in \widehat{C_G}(x)$, then the rules in Lemma \ref{l:1} allow us to conclude that
\[ab^{-1}\widehat{\wedge} x=(b^{-1} \widehat{\wedge} x)^a \  (a \widehat{\wedge}
x)={((b \widehat{\wedge} x)^{-1})}^{ab^{-1}}  \ (a \widehat{\wedge} x)=1_{G \widehat{\wedge} G}\]
so $ab^{-1}\in \widehat{C_G}(x)$. The \textit{complete exterior center} of $G$ is the set \[\widehat{Z}(G)=\{g \in G \ | \ 1_{_{G \widehat{\wedge} G}}=g \widehat{\wedge} y \in G \widehat{\wedge} G, \forall y \in G\}\] 
and a similar argument shows that it is a subgroup of $G \widehat{\wedge} G$. Analogously,  $\widehat{C_G}(x)$ is a subgroup of $C_G(x)$ and  $\widehat{Z}(G)$ is a subgroup of $Z(G)$.

Of course, if $G$ is a finite group,  $\widehat{\wedge}$ is exactly the  nonabelian exterior square $\wedge$ in \cite{brown} and $\widehat{\otimes}$ is the nonabelian tensor square in $\otimes$.  Then $\widehat{C_G}(x)=C^\wedge_G(x)$ and $\widehat{Z}(G)=Z^\wedge(G)$ so we have a significant approach in order to generalize most of the results in \cite{brown,eick2,n1}.

\begin{lemma}\label{l:2} 
Let $G$ be  a pro--$p$--group,  $A$ and $B$ two closed subgroups of $G$ such that $A \cap B=\{1\}$,  $N$ a closed normal subgroup of $G$ and $x \in G$. Then
\begin{itemize}
\item[(i)]$\widehat{C_G}(x)$ is a closed normal subgroup of $C_G(x)$;
\item[(ii)]$\widehat{Z}(G)={\underset{g\in G}\bigcap}\widehat{C_G}(g)$ is a closed subgroup of $Z(G)$;
\item[(iii)] $1 \longrightarrow \widehat{C_G}(x)\cap N
\longrightarrow \widehat{C_G}(x) \longrightarrow
\widehat{C_{G/N}}(xN)$ is a short exact sequence of pro--$p$--groups;
\item[(iv)]$1 \longrightarrow \widehat{Z}(G)
\cap N \longrightarrow \widehat{Z}(G) \longrightarrow \widehat{Z}
(G/N)$ is a short exact sequence of pro--$p$--groups;
\item[(v)]$\widehat{C_{A \times B}}(ab)=\widehat{C_A}(a) \times
\widehat{C_B}(b)$ for all $a \in A$ and $b \in B$; 
\item[(vi)]$\widehat{Z}(A \times B)=\widehat{Z}(A) \times
\widehat{Z}(B)$.
\end{itemize}
\end{lemma}

\begin{proof}(i). If $g\in \widehat{C_G}(x)$ and $y \in \widehat{C_G}(x)$, then 
$g^y\widehat{\wedge} x=(g\widehat{\wedge} x)^y=1_{G \widehat{\wedge}G}$
thus $\widehat{C_G}(x)$  is normal in $C_G(x)$. The fact that $\widehat{C_G}(x)$ is the stabilizer of a point, under the action of the operator $\widehat{\wedge}$, ensures that it is closed.

(ii). The equality of $\widehat{Z}(G)$ with the intersection of the complete exterior centralizers is obvious. It is obvious also that $\widehat{Z}(G)$  is contained in $Z(G)$. Finally, the intersection of closed is closed, then $\widehat{Z}(G)$  is  closed in $Z(G)$.

(iii). Consider the natural epimorphism of pro--$p$--groups \[\widehat{\pi}: g\wedge h \in G\widehat{\wedge} G\mapsto gN \widehat{\wedge} hN \in G/N \widehat{\wedge} G/N.\] If $y \in \widehat{C_G}(x)$, then $\widehat{\pi}(y)=yN \in
\widehat{C_{G/N}}(xN)$. On the other hand, if $\widehat{\pi}(y)=1$ for $y \in
\widehat{C_G}(x)$, then $y \in N\cap \widehat{C_G}(x)$. The result
follows.

(iv). It follows from (ii) and (iii).

(v). We may overlap  \cite[Proof of Proposition 2.6]{n1}, mutatis mutandis. 

(vi). It follows from (ii) and (v).
\end{proof}

Now, we may give a structural result which is related to $H_2(G,\mathbb{Z}_p)$.

\begin{proposition}\label{p:1} Let $G$ be a pro--$p$--group.
$C_G(x)/\widehat{C_G}(x)$ is isomorphic as pro--$p$--group with  a subgroup of $H_2(G,\mathbb{Z}_p)$.
\end{proposition}

\begin{proof} The map \[\widehat{\psi}: y \in C_G(x) \longmapsto
\widehat{\psi}(y)=x\widehat{\wedge} y \in H_2(G,\mathbb{Z}_p)\] is a homomorphism of pro--$p$--groups. Note that
$H_2(G,\mathbb{Z}_p)$ is a factor group of $\widehat{\nabla}(G)$, which is  a closed central
subgroup of $G \widehat{\otimes} G$. Furthemore, the elements of $\widehat{\nabla}(G)$ are
fixed under the action of $G$. This allows us to conclude that the
elements of $H_2(G,\mathbb{Z}_p)$ are fixed by the action of $G$. On the other
hand, it is clear that $\ker \widehat{\psi}= \widehat{C_G}(x)$.
Then the result follows. \end{proof}

An interesting consequence is listed below.

\begin{corollary}\label{c:1}If a pro--$p$--group $G$ has  an element $x$ such that $C_G(x)\not=\widehat{C_G}(x)$, then $H_2(G,\mathbb{Z}_p)$ is nontrivial.
\end{corollary}

\section{Main theorems}
The present section illustrates our main results. Firstly, we show an upper bound for $H_2(G,\mathbb{Z}_p)$, when $G$ is a pro--$p$--group. This is in  harmony with the results in \cite{eick2, ahmad5, moravec}.

\begin{theorem}\label{t:1}Let $G$ be a pro--$p$--group,  $\mathrm{rk}(G/\widehat{Z}(G))=n $, $\mathrm{rk}(H_2(G,\mathbb{Z}_p))=m $ and  $\mathrm{tf}(H_2(G,\mathbb{Z}_p))=l$ . Then $H_2(G,\mathbb{Z}_p)$ is an abelian pro--$p$--group. Furthermore,
\begin{itemize}
\item[(i)] if  $H_2(G,\mathbb{Z}_p)$ is finite, then $|Z(G)/\widehat{Z}(G)|$ divides
$|H_2(G,\mathbb{Z}_p)|^n$;
\item[(ii)] if  $H_2(G,\mathbb{Z}_p)$ is infinite, then $\mathrm{rk}(Z(G)/\widehat{Z}(G)) \le m^n$. In particular, if $H_2(G,\mathbb{Z}_p)$ is torsion--free, then $\mathrm{tf}(Z(G)/\widehat{Z}(G)) \le l^n$.
\end{itemize}
\end{theorem}

\begin{proof} 
From $(**)$,  $H_2(G,\mathbb{Z}_p)$ is a closed central subgroup of $G \widehat{\wedge} G$ and so it is a closed abelian subgroup of $G \widehat{\wedge} G$. From  Theorem \ref{primoz},  $G \widehat{\wedge} G$ is a pro--$p$--group and each closed subgroup of a pro--$p$--group is again a pro--$p$--group \cite[Lemma 7.2.9 (i)]{charles}. Then $H_2(G,\mathbb{Z}_p)$ is an abelian pro--$p$--group.

Now we proceed to prove (i) and (ii). Assume that $G/\widehat{Z}(G)=\overline{\langle \bar{x}_1, \ldots
\bar{x}_n\rangle}$ for some elements
$\bar{x}_1=x_1 \widehat{Z}(G), \ldots, \bar{x}_n=x_n \widehat{Z}(G)$
of $G/\widehat{Z}(G)$. Define
\[\widehat{\xi}: x\in Z(G)\longmapsto (x\ \widehat{\wedge} x_1, \ldots, x \widehat{\wedge} x_n)\in
H_2(G,\mathbb{Z}_p)^n.\]  $\widehat{\xi}$ is a homomorphism of pro--$p$--groups,
because for all $x, y\in Z(G)$  we have
\[xy \widehat{\wedge} x_i=(x \widehat{\wedge} x_i) \, \, (y \widehat{\wedge} x_i)^x=(x \widehat{\wedge} x_i) \, \, (y \widehat{\wedge} x_i)\] for
every $i\in \{1,\ldots n\}$.

We claim that $\ker \widehat{\xi}=\widehat{Z}(G)$. It is easy to
check that $\widehat{Z}(G)\leq \ker \widehat{\xi}$. On the other hand,
if $x \in \ker \widehat{\xi}$, then $x \widehat{\wedge} x_i =1_{G \widehat{\wedge} G}$ for every $i\in
\{1,\ldots n\}$. It is enough to show that $x \widehat{\wedge} y =1_{G \widehat{\wedge} G}$ for
every $y \in G$ in order to finish our proof. If $y \in G
\setminus \widehat{Z}(G)$, then we may always write 
$y=x^{\alpha_1}_1\ldots x^{\alpha_n}_n$, where $\alpha_1, \ldots
\alpha_n$ are integers. Thus, \[x \widehat{\wedge} y = x \widehat{\wedge}
(x^{\alpha_1}_1\ldots x^{\alpha_n}_n) = (x \widehat{\wedge}
x^{\alpha_1}_1)\ldots (x \widehat{\wedge} x^{\alpha_n}_n)\]
\[ = (x \widehat{\wedge}
x_1)^{\alpha_1}\ldots (x \widehat{\wedge} x_n)^{\alpha_n} = 1_{G \widehat{\wedge} G}.\]
We have then proved that $\widehat{\xi}$ is a monomorphism of pro--$p$--groups. Then $\widehat{\xi}$ allows us to embed $Z(G)/\widehat{Z}(G)$ in  $H_2(G,\mathbb{Z}_p)^n$. In case (i) this means that, if 
$H_2(G,\mathbb{Z}_p)$ is finite, then $H_2(G,\mathbb{Z}_p)^n$ is finite and consequently $Z(G)/\widehat{Z}(G)$ is finite, subject to the corresponding arithmetic condition in (i). In case (ii) $H_2(G,\mathbb{Z}_p)$ is infinite abelian of finite (topological) rank. Consequently, each closed subgroup of $H_2(G,\mathbb{Z}_p)$ has the same property. We conclude that (ii) is true. 
\end{proof}

\begin{remark}It is well known that a finite group $E$ is capable if and only if $Z^\wedge(E)=1$. Then the previous result may be useful to find criteria for the size of $\widehat{Z}(G)$ in a pro--$p$--group $G$. This aspect has been studied in \cite{bfs, n1} in the finite case.\end{remark}

In order to generalize to the infinite case the exterior degree, we proceed in the following way. Let $G$ be a pro--$p$--group. On the pro--$p$--group $G \times G$, we may consider the \textit{crossing pairing}
\[\widehat{f} : (x, y) \in G \times G  \longmapsto x \widehat{\wedge} y \in G \widehat{\wedge} G,  \]
described in \cite[\S 2]{moravec}, which is continuous and bilinear for all $x,y \in G$.  This map $\widehat{f}$ appears, for instance, in the  \textit{universal property} of $G \widehat{\wedge} G$. On the other hand, we know  that there exists a unique normalized Haar measure $\mu$ on a pro--$p$--group $G$ (the reader may find also a more general situation in \cite{ahmad2, ahmad3, ahmad4, rashid1, rashid2})  and this is true also for $G \times G$, when we consider the product measure  $\mu \times \mu$. Then
\[C=\widehat{f}^{-1}(1_{G \widehat{\wedge} G})=\{(x,y) \in G \times G \ | \ x \widehat{\wedge} y = 1_{G \widehat{\wedge} G}\} \subseteq G \times  G\] 
is a closed subgroup of $G \times G$ and so it is meaningful to define the \textit{exterior degree of a pro--$p$--group} $G$ as 
\[\widehat{\mathrm{d}}(G)=(\mu \times \mu) (C).\]
In case $G$ is finite and $\mu$ is the counting measure on $G$, we find  exactly $(*)$.  

The next lemmas show that some methods of \cite{ahmad1, ahmad4,  n2, rashid1, rashid2} can be adapted here. The idea is in fact to relate  $\widehat{\mathrm{d}}(G)$ with $\mathrm{d}(G)$.

\begin{lemma}\label{l:3}Let $G$ be a pro--$p$--group and $x \in G$.
Then \[ \widehat{\mathrm{d}}(G)= \int_G \mu (\widehat{C_G}(x)) d\mu (x),\] where
\[\mu (\widehat{C_G}(x))=\int_{G \times G} \chi_{_C}(x,y)
d\mu (y),\] and $\chi_{_C}$ denotes the characteristic map of $C$.
\end{lemma}

\begin{proof}  Fubini-Tonelli's Theorem implies:
\[
\widehat{\mathrm{d}}(G)=(\mu \times \mu )(C)= \int_{G \times G}
 \chi_{_{C}}(d\mu \times d\mu)\]
 \[=\int_G \left( \int_G  \chi_{_C}(x,y) d\mu(x) \right) d\mu(y)= \int_G \mu (\widehat{C_G}(x)) d\mu (x).\]
 \end{proof}

\begin{lemma} \label{l:4} Let $H$ be a closed subgroup of a
pro--$p$--group $G$ and $k\ge1 $. Then
\[\mu (H)=\left\{\begin{array}{lcl} \frac{1}{p^k},& \, if \, & |G:H|=p^k\\
0,& \, if \, & |G:H|=\infty.\end{array}\right.\]\end{lemma}

\begin{proof}  It is a well--known fact, which can be found for instance in  \cite{ahmad4}. \end{proof}

\begin{lemma}\label{l:5} A pro--$p$--group $G$ has  $0\le \widehat{\mathrm{d}}(G) \le 1$. In particular, \begin{itemize}
\item[(i)]$\widehat{\mathrm{d}}(G)=0$ if and only if $|G:\widehat{C_G}(x)|=\infty$ for all but finitely many $x \in G$;
\item[(ii)]$\widehat{\mathrm{d}}(G)=1$ if and only if $\widehat{Z}(G)=G$.
\end{itemize}
\end{lemma}
\begin{proof}  Since $\mu$ is monotone, positive--defined and normalized, 
$0=(\mu \times \mu) (\{(1,1)\}) \le \widehat{\mathrm{d}}(G)=(\mu \times \mu) (C) \le (\mu \times \mu) (G)=1.$ Now (i) follows easily from Lemmas \ref{l:3} and  \ref{l:4}.
(ii) is clear from the definition of $\widehat{\mathrm{d}}(G)=1$.
\end{proof}

A fundamental difference with the finite case is the following.

\begin{remark} $\mathbb{Z}_p$ is an infinite abelian pro--$p$--group topologically generated by 1 element (but not generated by 1 element in the abstract sense) such that $\widehat{Z}(\mathbb{Z}_p)=\mathbb{Z}_p$, $\widehat{\mathrm{d}}(\mathbb{Z}_p)=1$ and $H_2(G,\mathbb{Z}_p)$ is trivial.
\end{remark}

We may show the main result of this section.

\begin{theorem}\label{fundamental} A pro--$p$--group $G$ satisfies the following inequality
\[\widehat{\mathrm{d}}(G)\leq \mathrm{d}(G)-\left(\frac{p-1}{p}\right)\left(\mu(Z(G))-\mu(\widehat{Z}(G))\right).\] Furthermore, if $H_2(G,\mathbb{Z}_p)$ is finite, then
\[\widehat{\mathrm{d}}(G) \geq \mu(\widehat{Z}(G))+\frac{1}{|H_2(G,\mathbb{Z}_p)|}\left(\mathrm{d}(G)-\mu(\widehat{Z}(G))\right).\]
\end{theorem}

\begin{proof} We begin to prove the upper bound. Let $x \not \in \widehat{Z}(G)$.  Using Lemma \ref{l:3}, 
\[\widehat{\mathrm{d}}(G)=\int_G \mu (\widehat{C_G}(x)) d\mu (x)\]
\[=\mu(\widehat{Z}(G))+\int_{Z(G)-\widehat{Z}(G)} \mu (\widehat{C_G}(x)) d\mu (x)+\int_{G-Z(G)} \mu (\widehat{C_G}(x)) d\mu (x)\]
but the monotonicity of $\mu$ implies $ \mu (\widehat{C_G}(x)) \le  \mu(C_G(x))$ and  Lemma \ref{l:4} implies $\mu(\widehat{C_G}(x))= |G: \widehat{C_G}(x)|^{-1} \leq \frac{1}{p}$, thus 
\[\leq\mu(\widehat{Z}(G))+\frac{1}{p}\left(\mu(Z(G))-\mu(\widehat{Z}(G))\right)+\int_{G-Z(G)} \mu (C_G(x)) d\mu (x)\]
\[=\mu(\widehat{Z}(G))+\frac{1}{p}\left(\mu(Z(G))-\mu(\widehat{Z}(G))\right)+\mathrm{d}(G)-\mu(Z(G))\]
\[=\mathrm{d}(G)-\left(\frac{p-1}{p}\right)\left(\mu(Z(G))-\mu(\widehat{Z}(G))\right).\]
For the lower bound,  Lemmas \ref{l:3}, \ref{l:4} and Proposition \ref{p:1} imply 
\[\widehat{\mathrm{d}}(G)=\int_G \mu (\widehat{C_G}(x)) d\mu (x)=\mu(\widehat{Z}(G))+\int_{G-\widehat{Z}(G)} \frac{\mu (C_G(x))}{|C_G(x):\widehat{C_G}(x)|} d\mu (x)\]
\[ \ge \mu(\widehat{Z}(G))+\frac{1}{|H_2(G,\mathbb{Z}_p)|}\int_{G-\widehat{Z}(G)} \mu (C_G(x))d\mu (x)\]
\[=\mu(\widehat{Z}(G))+\frac{1}{|H_2(G,\mathbb{Z}_p)|}\left(\int_{G} \mu (C_G(x))d\mu (x)-\mu(\widehat{Z}(G))\right)\]
\[=\mu(\widehat{Z}(G))+\frac{1}{|H_2(G,\mathbb{Z}_p)|}\left(\mathrm{d}(G)-\mu(\widehat{Z}(G))\right).\]
\end{proof}

Theorem \ref{fundamental} allows us to conclude another interesting inequality. 

\begin{corollary}\label{unidegree} A pro--$p$--group $G$ satisfies always the inequality
$\widehat{\mathrm{d}}(G)\leq \mathrm{d}(G)$. In particular,   $\widehat{\mathrm{d}}(G)=\mathrm{d}(G)$  implies $\widehat{Z}(G)=Z(G)$.
\end{corollary}


For instance,  $\widehat{\mathrm{d}}(G)= \mathrm{d}(G)$, whenever $H_2(G,\mathbb{Z}_p)$ is trivial. This happens for the infinite pro--2--group (with $r \ge1 $ arbitrary) \[D=\langle a, t \  | \ a^{2^r} = 1, a^{-1}ta = t^{-1} \rangle = \mathbb{Z}_2 \rtimes C_{2^r},\]  described also in \cite[\S 1]{eick2}. Its exterior degree is computed below.

\begin{example}
 It is clear that $Z(D)=\widehat{Z}(D)=1$. For $i=0$, $\mu(\widehat{C_D}(t^i))=1$; for all $i\not=0$, $\mu(\widehat{C_D}(t^i))=1/2^r$; for all $i$ and $1\leq j\leq 2^r-1$, $\mu(\widehat{C_D}(a^jt^i))=0$.  By Lemma \ref{l:3} we find that \[\widehat{\mathrm{d}}(D)=\mu(\widehat{Z}(D))+\int_{\langle t\rangle-\widehat{Z}(D)} \mu(\widehat{C_D}(x)) d\mu(x)+\int_{D-\langle t\rangle}\mu(\widehat{C_D}(x)) d\mu(x)\] \[=\frac{1}{2^r}\mu(\langle t\rangle-\{1\})=\frac{1}{2^r}\mu(\langle t\rangle)=\frac{1}{4^r}.\]
\end{example}

The importance of the  condition $\widehat{\mathrm{d}}(G)= \mathrm{d}(G)= \mathrm{d}^\wedge(G)$  is illustrated in  \cite{n2} in finite case. The following result is an application of Theorem \ref{fundamental} and, at the same time, is a generalization of the corresponding results of \cite{n2} to the infinite case.

\begin{corollary}\label{application} Assume that $G$ is a pro--$p$--group. 
\begin{itemize}
\item[(i)] If $G$ is abelian of $\mathrm{rk}(G)>1$,  then $\widehat{\mathrm{d}}(G)\leq \frac{p^2+p-1}{p^3}$ and the equality holds if and only if $G/\widehat{Z}(G) \simeq C_p \times C_p$.
\item[(ii)] If $G$ is nonabelian and $\widehat{Z}(G)\not=Z(G)$, then $\widehat{\mathrm{d}}(G)\leq \frac{p^3+p-1}{p^4}.$
\end{itemize}
\end{corollary}

\begin{proof} (i). Abelian pro--$p$--groups of $\mathrm{rk}(G)=1$ are procyclic pro--$p$--groups and they are either isomorphic to $C_{p^k}$ (for some $k\ge1$) or to $\mathbb{Z}_p$. Both of them have $\widehat{Z}(G)=Z(G)=G$ by Lemma \ref{l:5}. Let $\widehat{Z}(G)\not=Z(G)=G$ and  $\mathrm{d}(G)=1$.
For all $x\not\in \widehat{Z}(G)$, arguing as in Theorem \ref{fundamental}, we find that $|G:\widehat{C_G}(x)|\geq p$, $|\widehat{C_G}(x):\widehat{Z}(G)|\geq p$ and $\mu(\widehat{Z}(G))\leq 1/p^2$. From Theorem \ref{fundamental},  we have
\[\widehat{\mathrm{d}}(G)\leq 1-\left(\frac{p-1}{p}\right)\left(1-\frac{1}{p^2}\right)=\frac{p^2+p-1}{p^3}.\]
Now assume $G/\widehat{Z}(G)\cong C_p \times C_p$. From Lemma \ref{l:3}, we get 
\[\widehat{\mathrm{d}}(G)=\int_G \mu (\widehat{C_G}(x)) d\mu (x)=\mu(\widehat{Z}(G))+\int_{G-\widehat{Z}(G)} \mu (\widehat{C_G}(x)) d\mu (x)\]
\[=\frac{1}{p}+\mu(\widehat{Z}(G)) \Big(1-\frac{1}{p}\Big)=\frac{p^2+p-1}{p^3}.\]
 Conversely, 
\[\frac{p^2+p-1}{p^3}=\widehat{\mathrm{d}}(G)=\int_G \mu (\widehat{C_G}(x)) d\mu (x)\leq\frac{1}{p}+\mu(\widehat{Z}(G))\Big(1-\frac{1}{p}\Big)\]
implies $\mu(\widehat{Z}(G))\geq 1/p^2$ and then $\mu(\widehat{Z}(G))=1/p^2$. We conclude necessarily that $G/\widehat{Z}(G)\cong C_p\times C_p.$

(ii). Assume that $G$ is nonabelian and $\widehat{Z}(G)\not=Z(G)$. Then we may argue as (i) above, getting $\mathrm{d}(G)\leq (p^2+p-1)/p^3$, $\mu(G)\leq 1/p^2$ and $\mu(\widehat{Z}(G))\leq 1/p^3$. Again an application of Theorem \ref{fundamental} allows us to conclude the proof.
\end{proof}

We end with two theorems which describe the exterior degree for quotients and direct products. The first case is the following.

\begin{theorem}\label{qut}
Let $N$ be a closed normal subgroup of a pro--$p$--group $G$. Then \[\widehat{\mathrm{d}}(G)\leq \widehat{\mathrm{d}}(G/N)\] and the equality holds if $N \le  \widehat{Z}(G)$.
\end{theorem}
\begin{proof}
Assume that $\lambda$, $\mu$ and $\nu$ are corresponding Haar measure of $N$, $G$ and $G/N$ respectively. The Extended Weil's Formula (see for instance \cite[Equation $(*)$, p. 126]{rashid2}), implies 
\[\widehat{\mathrm{d}}(G)=\int_G \mu (\widehat{C_G}(x)) d\mu (x)\leq\int_G \mu (\widehat{C_G}(x)N)d\mu (x)\]
\[=\int_{\frac{G}{N}}\int_{N} \mu (\widehat{C_G}(xn)N)d\lambda(n)d\nu(xN).\]
On the other hand, one can see  without difficulties that
\[\mu(\widehat{C_G}(xn)N)=\nu(\widehat{C_G}(xn)N/N)\leq\nu(\widehat{C_{G/N}}(xN))\]
therefore,
\[\widehat{\mathrm{d}}(G)\leq\int_{\frac{G}{N}}\int_{N}\nu(\widehat{C_{G/N}}(xN))d\lambda(n)d\nu(xN)\]
\[=\int_{\frac{G}{N}}\nu(\widehat{C_{G/N}}(xN))\left(\int_{N}d\lambda(n)\right)d\nu(xN)=\int_{\frac{G}{N}}\nu(\widehat{C_{G/N}}(xN))d\nu(xN)\]
\[=\widehat{\mathrm{d}}(G/N).\]
Now assume that $N$ is contained in $\widehat{Z}(G)$, then the canonical map $G\widehat{\wedge}G \rightarrow G/N\widehat{\wedge}G/N $ will be an isomorphism. Therefore $\widehat{C_G}(xn)N/N=\widehat{C_{G/N}}(xN)$ for all $x\in G$ and all inequalities should be changed into equalities and so
$\widehat{\mathrm{d}}(G)=\widehat{\mathrm{d}}(G/N).$ 
\end{proof}

The second case describes a theorem of splitting of probabilities.

\begin{theorem}\label{product}
Let $G$ be a pro--$p$--group  and $H$ be a pro--$q$--group for two different primes $p$ and $q$. Then \[\widehat{\mathrm{d}}(G\times H)=\widehat{\mathrm{d}}(G) \ \ \widehat{\mathrm{d}}(H).\]
\end{theorem}
\begin{proof}
Assume that $\mu$ and $\nu$ are the Haar measure on $G$ and $H$ respectively. Then $\lambda=\mu\times\nu$ will be the Haar measure on $G\times H$. Forevermore, for all $(x,y)\in G\times H$, we have $\widehat{C_{G\times H}}((x,y))=\widehat{C_G}(x)\times\widehat{C_H}(y)$, then
\[\widehat{\mathrm{d}}(G\times H)=\int_{G\times H} \lambda(\widehat{C_{G\times H}}((x,y))) d\lambda ((x,y))\]
\[=\int_G\int_H \mu(\widehat{C_G}(x))\nu(\widehat{C_H}(y))d\mu(x)d\nu(y)\]
\[\int_G\mu(\widehat{C_G}(x))d\mu(x)\int_H \nu(\widehat{C_H}(y))d\nu(y)=\widehat{\mathrm{d}}(G) \ \ \widehat{\mathrm{d}}(H).\]
\end{proof}

\end{document}